\documentclass[11pt]{article}

\usepackage[T2A]{fontenc}
\usepackage[cp1251]{inputenc}
\usepackage{amsthm}%%%%%%%% theorem definition
\usepackage{amsfonts}
\usepackage{amssymb}
\usepackage{amsmath}
\usepackage{cite}%%%%% citations

\usepackage{bm}
 
\begin{document}
%%%%%%%%%%%%% begin theorem definition %%%%%%%%%%%%%%%%%%
\newtheoremstyle{mytheorem}
  {\topsep}   % ABOVESPACE
  {\topsep}   % BELOWSPACE
  {\itshape}  % BODYFONT
  {}       % INDENT (empty value is the same as 0pt)
  {\bfseries} % HEADFONT
  { }         % HEADPUNCT
  {5pt plus 0pt minus 1pt} % HEADSPACE
  { }          % CUSTOM-HEAD-SPEC
\newtheoremstyle{myremark}
  {\topsep}   % ABOVESPACE
  {\topsep}   % BELOWSPACE
  {\upshape}  % BODYFONT
  {}       % INDENT (empty value is the same as 0pt)
  {\bfseries} % HEADFONT
  {  }         % HEADPUNCT
  {5pt plus 1pt minus 1pt} % HEADSPACE
  { }          % CUSTOM-HEAD-SPEC\cite{}
\theoremstyle{mytheorem}
\newtheorem{theorem}{Theorem}[section]
 \newtheorem*{heyde1*}{Theorem A}
  \newtheorem*{heyde2*}{Theorem B}
 \newtheorem{proposition}[theorem]{Proposition}
\newtheorem{lemma}[theorem]{Lemma}
\newtheorem{corollary}[theorem]{Corollary}
\theoremstyle{myremark}
\newtheorem{remark}[theorem]{Remark}
\newtheorem{problem}[theorem]{Problem}
%%%%%%%%%%%%%%%%%%%%% end theorem definition %%%%%%%%%%%%%%%%%%
\noindent This article is accepted for publishing in

\noindent Journal of Mathematical Analysis and Applications

\vskip 1 cm
\noindent\textbf{\Large Heyde characterization theorem for some} 

\medskip

\noindent\textbf{\Large  classes of locally compact Abelian groups}

\bigskip

\noindent\textbf{\large Gennadiy Feldman}  

\bigskip

\noindent ORCID ID  https://orcid.org/0000-0001-5163-4079 

\bigskip

\noindent{\textbf{Abstract}} 

\noindent Let $L_1$ and $L_2$ be linear forms of real-valued 
independent random variables. 
By Heyde's theorem, if the conditional distribution of $L_2$ given $L_1$ is symmetric,
then the random variables are Gaussian. A number of papers are devoted 
to generalisation of Heyde's theorem to the case, where independent random variables take 
values in a locally compact Abelian group  $X$. 
The article continues these studies. We consider the case, where $X$ is either a 
totally disconnected group or is of the form
$\mathbb{R}^n\times G$, where $G$ is a totally disconnected 
group  consisting of compact elements. The proof is based on the study of solutions 
of the Heyde functional equation on the character group of the original group.  
In so doing, we use methods of abstract harmonic analysis.

\bigskip

\noindent{\textbf{Key words} Locally compact Abelian group $\cdot$    
Topological automorphism  $\cdot$ Functional equation $\cdot$  
Characterization theorem  $\cdot$  Conditional distribution

\bigskip

\noindent{\textbf{Mathematical Subject Classification (2020)}   39B52 $\cdot$ 43A35  
$\cdot$  60B15  $\cdot$  62E10

\section{Introduction and basic notation}

Let $\alpha_i$, $\beta_i$, $i=1, 2$, 
be   nonzero real numbers satisfying the condition 
$\beta_1\alpha_1^{-1} + \beta_2\alpha_2^{-1}\ne 0$. 
 Let $\xi_1$ and $\xi_2$ be independent random variables. 
 Consider the linear forms $L_1 = \alpha_1\xi_1 +\alpha_2\xi_2$ and
$L_2=\beta_1\xi_1+\beta_2\xi_2$ and suppose that the conditional distribution of
$L_2$ given $L_1$ is symmetric\footnote{It should be noted that
if $\xi$ and $\eta$ are random variables, then  the conditional  
distribution of  $\eta$ given $\xi$ is symmetric
if and only if the random vectors $(\xi, \eta)$ and $(\xi, -\eta)$ are 
identically distributed.}.
By the well-known Heyde  theorem (\!\!\cite{He}, see also 
 \cite[Theorem 13.4.1]{KaLiRa}), $\xi_1$ and $\xi_2$ are Gaussian.

This result stands in line with such classical statements 
as the Kac--Bernstein theorem, where the Gaussian distribution is characterized
by the independence of the sum and difference of independent random variables, and 
the Skitovich--Darmois theorem, where the Gaussian distribution is characterized
by the independence of two linear forms of $n$ independent random variables.
Assertions of this type are called characterization theorems of mathematical statistics. 
The important fact is that for above mentioned theorems, the proof reduces to  
solving 
some functional equations. In the case when random variables take values in a 
locally compact Abelian group  $X$, we get  functional equations 
on the character group of
$X$. These equations and others like them are of great interest 
and have been studied
by many authors 
independently of characterization problems. We also note that 
polynomials on groups play an important role in solving the corresponding 
functional equations
(see, e.g., J.M.~Almira \cite{A1}, \cite{A2}, 
J.M.~Almira and E.V.~Shulman \cite{AS1}, 
M.~Sablik and E.V.~Shulman \cite{SaSh1}, E.V.~Shulman \cite{Sh1}, 
H.~Stetkaer \cite{St1}, \cite{St2}, L.~Sz\'ekelyhidi \cite{Sz1}).

A number of papers are devoted to analogues of Heyde's theorem in 
the case, where independent random variables take values in a locally 
compact Abelian group  $X$ 
(see, e.g., \cite{F_solenoid, 
M2013, H1, H2, JMAA2024} and also  \cite[Chapter IV]{Febooknew}, 
where one can 
find additional references). 
  In doing so, coefficients of the  linear forms are topological 
  automorphisms of $X$. In most articles, the coefficients satisfy 
  some conditions, which can be considered as group 
  analogues of the  condition on the real line. In the articles  
  \cite{F_solenoid,   H1, H2} an analogue  of Heyde's theorem 
  was proved without any restrictions on  the coefficients 
  of the linear forms. In this article we continue these investigations. 
  
Let $\xi_1$ and $\xi_2$ be
independent random variables with values in a locally compact Abelian group  
$X$ and distributions
$\mu_1$ and $\mu_2$. Let $\alpha_j$, $\beta_j$ be topological 
  automorphisms of $X$.
 Assume that  the conditional  distribution of the linear form  
 $L_2 = \beta_1\xi_1  + \beta_2\xi_2$ given $L_1 = \alpha_1\xi_1 +
\alpha_2\xi_2$    is symmetric.
It is easy to see that studying possible distributions $\mu_j$ 
  we can assume without  loss of generality that  
   $L_1=\xi_1+\xi_2$ and $L_2=\xi_1+\alpha\xi_2$, where $\alpha$ is a 
   topological automorphism of $X$.  The following theorem was proved in   
   \cite{H1}, see also \cite[Theorem 9.11]{Febooknew}.
 \begin{heyde1*}  Consider a finite Abelian group $X$  and  let  
 $\alpha$ be an automorphism of $X$. Assume that $X$ contains no elements 
 of order  $2$.
 Let $\xi_1$ and $\xi_2$ be
independent random variables with values in the group  $X$ and distributions
$\mu_1$ and $\mu_2$. Assume that  the characteristic functions 
$\hat\mu_1(y)$ and $\hat\mu_2(y)$ do not vanish.  Put
$L_1 = \xi_1 + \xi_2$ and $L_2 = \xi_1 + 
\alpha\xi_2$.
If the conditional  distribution of $L_2$ given $L_1$ is symmetric, then
$\mu_j$ are shifts of a distribution $\omega$ supported in the subgroup
$\mathrm{Ker}(I+\alpha)$. Furthermore, if
$\eta_j$ are independent identically distributed random variables with values in    
the group $X$  and distribution  $\omega$, then    the conditional  
distribution of the linear form
$N_2=\eta_1 + \alpha\eta_2$ given $N_1=\eta_1 + \eta_2$ is symmetric.
\end{heyde1*}
The purpose of the article is to prove an analogue of Heyde's theorem 
for a wide class of locally compact Abelian groups, without 
imposing any restrictions on the topological automorphism $\alpha$.
The main results of the article are as follows. In Section 2 we 
prove that Theorem A is true for a much wider 
class of groups, namely, for all totally disconnected locally compact 
Abelian groups   containing no elements of order 2 (Theorem \ref{th3}). 
In particular, 
it is true for the additive group of the 
field of $p$-adic numbers $\mathbb{Q}_p$. 
Based on Theorem \ref{th3}, in Section 3 we   prove an analogue of Heyde's 
theorem for independent random variables with values in the group    
$\mathbb{R}^n\times G$, where $G$ is a totally disconnected locally compact 
Abelian group consisting of compact elements and containing no elements 
of order 2 (Theorem \ref{th2}). In all the listed theorems we do not 
impose any restrictions on the topological automorphism $\alpha$. 
In this case the kernel $\mathrm{Ker}(I+\alpha)$
plays a critical role.

We will use in the article the standard results and concepts of abstract harmonic 
analysis (see, e.g., \cite{Hewitt-Ross}) and as well as 
some facts of complex analysis. Let $X$ be a   locally compact 
Abelian group with character group $Y$. We suppose that $X$ is a 
second countable group and consider only such groups without specifying it.  
Let $x\in X$ and $y \in Y$. Denote  by  $(x,y)$ 
the value of the character $y$ at the element  $x$.
Let $H$ be a closed subgroup of the group   $Y$.    Denote by  
$A(X, H) = \{x \in X: (x, y) =1$ for all $y\in
H\}$ the annihilator of $H$. If $\alpha:X\rightarrow X$
is a continuous endomorphism of the group  $X$, we define the adjoint endomorphism
$\widetilde\alpha: Y\rightarrow Y$ as follows  $(\alpha x,
y)=(x, \widetilde\alpha y)$ for all $x\in X$, $y\in
Y$. The group
of all topological automorphisms of the group $X$ we denote by $\mathrm{Aut}(X)$. 
Denote by  $I$ the identity automorphism of a group.   
Observe that $\alpha$ is a topological automorphism of the group   
$X$ if and only if $\widetilde\alpha$  is a topological automorphism 
of the group $Y$. Let $K$ be a closed subgroup of the group $X$ and $\alpha$
be a topological automorphism   
of the group  $X$ such that $\alpha(K)=K$. 
Then the restriction of  $\alpha$ 
 to $K$ is a topological automorphism of the subgroup 
$K$. Denote by  $\alpha_{K}$ this restriction. 
A closed subgroup $G$ of the group $X$ is called characteristic if
$\alpha(G)=G$ for all $\alpha\in\mathrm{Aut}(X)$.
  Put $X^{(2)}=\{2x: x\in X\}$.      
Denote by $\mathbb{C}$ the set of complex numbers.  
If $x, y\in \mathbb{R}^m$, we set   $\langle x, y\rangle=\sum_{j=1}^mx_jy_j$.
We also use this notation   in the case when     $x, y\in \mathbb{C}^m$.

Assume that $f(y)$ is an arbitrary function on  $Y$ and let $h\in Y$.
Define the finite difference operator $\Delta_h$  as follows
$$\Delta_h f(y)=f(y+h)-f(y), \quad y\in Y.$$
A function $f(y)$ on the group    $Y$
is called a polynomial if  $f(y)$
satisfies the equation
$$
 \Delta_{h}^{n+1}f(y)=0, \quad y,h \in Y,
$$
for a nonnegative integer $n$.

Let $\mu$ and $\nu$ be  probability
distributions  on the group $X$. The convolution
$\mu*\nu$ is defined by the formula
$$
\mu*\nu(A)=\int\limits_{X}\mu(A-x)d \nu(x)
$$
for any Borel set $A$. 
The characteristic function of the distribution  $\mu$
is defined as follows
$$
\hat\mu(y) =
\int\limits_{X}(x, y)d \mu(x), \quad y\in Y.$$ 
The support of $\mu$ is denoted by $\sigma(\mu)$.  
For any distribution $\mu$ we will write $\bar \mu $ for the distribution defined by 
$\bar \mu(A) = \mu(-A)$ for any Borel set $A$. 
%Then $\hat{\bar{\mu}}(y)=\overline{\hat\mu(y)}$. 
Let $x\in X$. The degenerate distribution
 concentrated at $x$ is denoted by $E_x$.

\section{ Heyde theorem  for totally disconnected locally compact Abelian groups}

In this section we prove an analogue of Heyde's theorem  for totally disconnected 
locally compact Abelian groups containing no elements of order  2. This theorem 
is a generalization of Theorem A. 
 \begin{theorem}\label{th3} Let $X$ be a totally disconnected locally compact 
 Abelian group containing no elements of order  $2$  and  let  $\alpha$ be a 
 topological automorphism of the group   $X$. Let $\xi_1$ and $\xi_2$ be
independent random variables with values in   $X$ and distributions
$\mu_1$ and $\mu_2$. Assume that  the characteristic functions 
$\hat\mu_1(y)$ and $\hat\mu_2(y)$ do not vanish.  Put
$L_1 = \xi_1 + \xi_2$ and $L_2 = \xi_1 + 
\alpha\xi_2$.
If the conditional  distribution of $L_2$ given $L_1$ is symmetric, then
$\mu_j=\omega*E_{x_j}$, where $\omega$ is   a distribution   supported in 
the subgroup
$\mathrm{Ker}(I+\alpha)$, $x_j\in X$, $j=1, 2$. Furthermore, if
$\eta_j$ are independent identically distributed random variables with values in    
the group $X$  and distribution  $\omega$, then    the conditional  
distribution of the linear form
$N_2=\eta_1 + \alpha\eta_2$ given $N_1=\eta_1 + \eta_2$ is symmetric.
\end{theorem}
To prove Theorem \ref{th3}, we need the following lemmas. We also 
use some of them in the proof of Theorem \ref{th2} in Section 3.

\begin{lemma} [{\protect\!\!\cite[Lemma 3.8]{M2013}}]\label{lem8}   Let   
$X$ be a locally compact Abelian group  and let  $\alpha$   be  a topological 
automorphism  of the group   $X$.
Let $\xi_1$ and $\xi_2$ be
independent random variables with values in   $X$.    
If   the conditional  distribution of the linear form 
$L_2 = \xi_1 + \alpha\xi_2$ given $L_1 = \xi_1 + \xi_2$
is symmetric, then  the linear forms
$M_1=(I+\alpha)\xi_1+2\alpha\xi_2$ and
$M_2=2\xi_1+(I+\alpha)\xi_2$ are independent.
\end{lemma}

\begin{lemma} [{\protect\!\!\cite[Lemma 6.1]{Febooknew}}]\label{lem7}  
 Let   $X$ be a locally compact Abelian group  and let  $\alpha_j$, 
 $\beta_j$, $j=1, 2$, be continuous endomorphisms of the group    $X$. 
 Let $\xi_1$ and $\xi_2$ be
independent random variables with values in   $X$ and distributions
$\mu_1$ and $\mu_2$. 
Then the following statements are equivalent:
\renewcommand{\labelenumi}{\rm(\roman{enumi})}
\begin{enumerate}
  
\item	

the linear forms $L_1 = \alpha_1\xi_1 +
\alpha_2\xi_2$ and $L_2 = \beta_1\xi_1  + \beta_2\xi_2$ are independent; 

\item

the characteristic functions $\hat\mu_j(y)$ satisfy the 
equation\footnote{If we consider linear 
forms $L_1 = \alpha_1\xi_1 +\dots+ \alpha_n\xi_n$ and 
$L_2=\beta_1\xi_1+\dots+\beta_n\xi_n$ of $n$ independent random 
variables, then the characteristic functions $\hat\mu_j(y)$ 
satisfy a similar equation  (\!\!\cite[Equation (6.1)]{Febooknew}).
 It is called the Skitovich--Darmois equation.}
\begin{equation}\label{18.04.1}
\hat\mu_1(\widetilde\alpha_1 u+\widetilde\beta_1 v)\hat\mu_2(\tilde\alpha_2 u
+\widetilde\beta_2 v)=\hat\mu_1(\widetilde\alpha_1
u)\hat\mu_2(\widetilde\alpha_2
u)\hat\mu_1(\widetilde\beta_1 v)\hat\mu_2(\widetilde\beta_2 v),
\quad u, v \in
Y.
\end{equation}
\end{enumerate} 
\end{lemma}
\begin{lemma}[{\protect\!\!\cite[(24.22)]{Hewitt-Ross}}]\label{lem9} 
 Let   
$X$ be a locally compact Abelian group.    The subgroup $Y^{(2)}$ 
is dense in the group $Y$  if and only if the group    $X$ contains no 
elements of order   $2$.
\end{lemma}
\begin{lemma}[{\protect see, e.g., \cite[Proposition  1.30]{Febooknew}}]\label{lem6}  
Let $Y$ be a locally compact Abelian group such that all elements of 
the group $Y$ are compact. Let $P(y)$ be a continuous polynomial on  $Y$. Then   
$P(y)=const$.
\end{lemma}
\begin{lemma}[{\protect\!\!\cite[Lemma 2.5]{H1}, see also 
\cite[Lemma 9.10]{Febooknew}}]\label{lem11}
 Let $X$ be a locally compact Abelian group  and
 let $H$ be a closed subgroup of the group  $Y$ such that the subgroup 
 $H^{(2)}$ is dense in   $H$.   
 Let  $\alpha$ be a topological automorphism of the group   $X$ such that 
 $\widetilde\alpha(H)=H$.
Let
$\xi_1$ and  $\xi_2$  be independent random variables with values in
 $X$  and distributions $\mu_1$ and $\mu_2$ such that
 \begin{equation}\label{19.04.1}
|\hat\mu_1(y)|=|\hat\mu_2(y)|=1, \quad y\in H.
\end{equation}
If the conditional  distribution of the linear form 
$L_2 = \xi_1 + \alpha\xi_2$ given $L_1 = \xi_1 + \xi_2$ is 
symmetric, then there are some shifts $\lambda_j$ of the 
distributions $\mu_j$  such that the distributions $\lambda_j$ are supported in 
$A(X, H)$. Furthermore, if $\eta_j$ are independent random variables with 
values in the group  $X$  and distributions $\lambda_j$, then    the 
conditional  distribution of the linear form
$N_2=\eta_1 + \alpha\eta_2$ given $N_1=\eta_1 + \eta_2$ is symmetric.
\end{lemma}
\begin{lemma}[{\protect\!\!\cite[Lemma 9.5]{Febooknew}}]\label{lem3} 
Let $X$ be a locally compact Abelian group containing no elements of order  $2$.
Let $\xi_1$ and $\xi_2$ be
independent random variables with values in   $X$ and distributions
$\mu_1$ and $\mu_2$.  The conditional  distribution of the linear form $L_2 = 
\xi_1 -\xi_2$  given $L_1 = \xi_1 + \xi_2$
is symmetric if and only if $\mu_1=\mu_2=\omega$, where $\omega$ 
is an arbitrary distribution.
\end{lemma}

\begin{lemma}[{\protect\!\!\cite[Lemma 9.1]{Febooknew}}]
\label{lem1} Let   $X$ be a locally compact Abelian group  and let  $\alpha_j$, 
 $\beta_j$, $j=1, 2$, be continuous endomorphisms of the group    $X$.
Let
$\xi_1$ and  $\xi_2$  be independent random variables with values in
 $X$  and distributions $\mu_1$ and $\mu_2$.  
Then the following statements are equivalent:
\renewcommand{\labelenumi}{\rm(\roman{enumi})}
\begin{enumerate}
  
\item	

the conditional  distribution of the linear form 
$L_2 = \beta_1\xi_1  + \beta_2\xi_2$  given $L_1 = \alpha_1\xi_1 +
\alpha_2\xi_2$ is symmetric; 

\item

the characteristic functions $\hat\mu_j(y)$ satisfy 
the equation\footnote{If we consider linear 
forms $L_1 = \alpha_1\xi_1 +\dots+ \alpha_n\xi_n$ and 
$L_2=\beta_1\xi_1+\dots+\beta_n\xi_n$ of $n$ independent random 
variables, then the characteristic functions $\hat\mu_j(y)$ 
satisfy a similar equation  (\!\!\cite[Equation (9.1)]{Febooknew}).
 It is called the Heyde equation.}
\begin{equation*}
\hat\mu_1(\widetilde\alpha_1u+\widetilde\beta_1v)
\hat\mu_2(\widetilde\alpha_2u+\widetilde\beta_2v)=
\hat\mu_1(\widetilde\alpha_1u-\widetilde\beta_1v)
\hat\mu_2(\widetilde\alpha_2u-\widetilde\beta_2v), \quad u, v \in Y.
\end{equation*}
In particular, if   $L_1 = \xi_1 + \xi_2$ and $L_2 = \xi_1 + \alpha\xi_2$,
where $\alpha$ is a continuous endomorphism of the group $X$,
this equation takes the  form
\begin{equation}\label{11.04.1}
\hat\mu_1(u+v )\hat\mu_2(u+\widetilde\alpha v )=
\hat\mu_1(u-v )\hat\mu_2(u-\widetilde\alpha v), \quad u, v \in Y.
\end{equation}
\end{enumerate} 
\end{lemma}

\textit{Proof of Theorem $\ref{th3}$} 
By Lemma  \ref{lem8}, the symmetry of the conditional  distribution 
of the linear form
 $L_2$ given
  $L_1$ implies that the linear forms
$M_1=(I+\alpha)\xi_1+2\alpha\xi_2$ and
$M_2=2\xi_1+(I+\alpha)\xi_2$ are independent.
By Lemma  \ref{lem7}, it follows from this that  the characteristic functions
$\hat\mu_j(y)$ satisfy equation $(\ref{18.04.1})$ which takes the form
\begin{equation}\label{18.04.2}
\hat\mu_1((I+\widetilde\alpha) u+2 v)\hat\mu_2(2\widetilde\alpha  
u+(I+\widetilde\alpha) v)=\hat\mu_1((I+\widetilde\alpha)
u)\hat\mu_2(2\widetilde\alpha
u)\hat\mu_1(2 v)\hat\mu_2((I+\widetilde\alpha) v), \quad u, v \in
Y.
\end{equation}
Put $\nu_j = \mu_j* \bar \mu_j$. Then  
$\hat \nu_j(y) = |\hat \mu_j(y)|^2 > 0$  for all $y \in Y$, $j=1, 2$.
Obviously, the characteristic functions
 $\hat \nu_j(y)$
also satisfy equation $(\ref{18.04.2})$. Set 
$$
P(y)=\log\hat \nu_1(y), \quad Q(y)=\log\hat\nu_2(y).
$$  
It follows from equation (\ref{18.04.2}) for the   characteristic functions 
$\hat\nu_j(y)$ that
\begin{equation}\label{18.04.3}
P((I+\widetilde\alpha) u+2 v)+Q(2\widetilde\alpha  
u+(I+\widetilde\alpha) v)=A(u)+B(v), \quad u, v \in
Y,
\end{equation}
where
\begin{equation}\label{18.04.4}
A(y)=P((I+\widetilde\alpha)y)+Q(2\widetilde\alpha y), 
\quad B(y)=P(2y)+Q((I+\widetilde\alpha) y), \quad y\in Y.
\end{equation}
Equation (\ref{18.04.3}) with certain restrictions on the topological 
automorphism  $\alpha$  arises in the study of    
Heyde's theorem on various locally compact Abelian groups. 
 We solve equation   (\ref{18.04.3}) using the finite 
difference method. This is a standard reasoning.
At the same time, it is important for us that we solve equation 
 (\ref{18.04.3}) on  the character group of a totally disconnected locally compact 
 Abelian group containing no elements of order  $2$.

Take an arbitrary element $h_1$ of the group $Y$.
Substitute $u+(I+\widetilde\alpha) h_1$ instead of $u$ 
and $v-2\widetilde\alpha  h_1$ instead   of $v$ into
 (\ref{18.04.3}). Subtracting equation
  (\ref{18.04.3}) from the resulting equation, we get
  \begin{equation}\label{18.04.5}
    \Delta_{(I-\widetilde\alpha)^2 h_1}{P((I+\widetilde\alpha) u+2 v)}
    =\Delta_{(I+\widetilde\alpha) h_1} A(u)+\Delta_{-2\widetilde\alpha  h_1} B(v),
\quad u, v\in Y.
\end{equation}
Take an arbitrary element $h_2$ of the group $Y$.
Substitute  $u+2h_{2}$ instead of $u$ and
$v-(I+\widetilde\alpha)h_{2}$ instead of $v$ into (\ref{18.04.5}). 
Subtracting equation
  (\ref{18.04.5}) from the obtained  equation, we find
 \begin{equation}\label{18.04.6}
     \Delta_{2 h_2}\Delta_{(I+\widetilde\alpha) h_1} 
     A(u)+\Delta_{-(I+\widetilde\alpha) h_2}\Delta_{-2\widetilde\alpha  h_1} B(v)=0,
\quad u, v\in Y.
\end{equation}
Take an arbitrary element $h_3$ of the group $Y$.
Substitute $u+h_3$ instead of $u$ into
(\ref{18.04.6}). Subtracting equation
  (\ref{18.04.6}) from the resulting  equation, we receive
 \begin{equation}\label{18.04.7}
   \Delta_{h_3}\Delta_{2 h_2}\Delta_{(I+\widetilde\alpha) h_1} A(u)=0,
\quad u\in Y.
\end{equation}

Put $H=\overline{(I+\widetilde\alpha)Y}$. We will prove that    
(\ref{19.04.1}) is fulfilled. The group $X$ contains no elements of order   
2. By Lemma \ref{lem9}, the subgroup $Y^{(2)}$ is dense in the group $Y$. 
Taking this into account 
and taking into account that
$h_j$ are arbitrary elements of the group  $Y$, it follows from
(\ref{18.04.7}) that the function   $A(y)$ satisfies the equation
$$
\Delta_h^{3} A(y) = 0, \quad y, h  \in H,
$$
i.e.,   $A(y)$ is a continuous polynomial on the subgroup    $H$. Since $X$ 
is a totally disconnected locally compact Abelian group, all elements 
of the group $Y$   are compact (\!\!\cite[(24.17)]{Hewitt-Ross}). 
Hence  all elements of the subgroup $H$ are also compact. Notice  that 
$A(0)=0$. By Lemma \ref{lem6}, $A(y)=0$ for all $y\in H$. Then 
(\ref{18.04.4}) implies that
\begin{equation}\label{09.05.1}
Q(2\widetilde\alpha y)=0, \quad y\in H.
\end{equation}

Put $K=\mathrm{Ker}(I+\alpha)$. The factor group $X/K$
contains no elements of order 2. Indeed, let $[x]\in X/K$ and $2[x]=0$. 
Then $[2x]\in K$, i.e., $(I+\alpha)2x=0$ and hence $2(I+\alpha)x=0$. 
Since the group $X$ contains no elements of order 2, this implies that
$(I+\alpha)x=0$. We obtain that $x\in K$ and $[x]=0$. It is obvious that   
$K=A(X, H)$ and hence $H=A(Y, K)$. The subgroup $H$ is topologically isomorphic
to the character group of the factor group $X/K$. Since $X/K$
contains no elements of order 2, by Lemma \ref{lem9}, the subgroup $H^{(2)}$ 
is dense in   $H$.
 Taking into account that $\widetilde\alpha$ is a topological automorphism 
 of the group  $Y$ and $\widetilde\alpha(H)=H$, it follows from (\ref{09.05.1})  
 that
$Q(y)=0$ for all $y\in H$. Hence  $|\hat\mu_2(y)|=1$   for all $y\in H$. 

Arguing similarly, we exclude from equation (\ref{18.04.6}) the function $A(y)$ 
and obtain that $B(y)=0$ for all $y\in H$. This implies that $P(2y)=0$ for 
all $y\in H$. Since the subgroup $H^{(2)}$ is dense in  $H$, 
it follows from this that $P(y)=0$ 
for all $y\in H$. Hence  $|\hat\mu_1(y)|=1$   for all $y\in H$. Thus, we 
proved that (\ref{19.04.1}) holds.

Due to the facts that  $K=A(X, H)$,  the subgroup $H^{(2)}$ is dense in   
$H$, and $\widetilde\alpha(H)=H$, we can apply Lemma \ref{lem11}.
By Lemma \ref{lem11}, there are some shifts $\lambda_j$ of the distributions 
$\mu_j$  such that $\lambda_j$ are supported in  the subgroup    $K$.  
Moreover,  if
$\eta_j$ are independent random variables with values in the group $X$  
and distributions $\lambda_j$, then    the conditional  distribution of 
the linear form
$N_2=\eta_1 + \alpha\eta_2$ given $N_1=\eta_1 + \eta_2$ is symmetric. 
Since  $K=\mathrm{Ker}(I+\alpha)$, the restriction of the automorphism 
$\alpha$  to the subgroup $K$ coincides with $-I$.
Since $\sigma(\lambda_j)\subset K$, we can consider $\eta_j$ as independent 
random variables with values in the group   $K$. In so doing, the conditional 
distribution of the linear form  $N_2=\eta_1 -\eta_2$ given $N_1=\eta_1 + \eta_2$  
is symmetric. It follows from Lemma \ref{lem3}  applying to the group   $K$  that 
$\lambda_1=\lambda_2=\omega$, where $\omega$ is an arbitrary distribution 
supported in $K$.
The theorem is completely proved. \hfill$\Box$

\medskip

Consider the field of $p$-adic numbers $\mathbb{Q}_p$. When we say 
the group $\mathbb{Q}_p$, we mean the additive 
group of the field $\mathbb{Q}_p$.
The group $\mathbb{Q}_p$ is a totally disconnected 
locally compact Abelian group containing no elements of order  $2$. 
Each topological 
automorphism of the group $\mathbb{Q}_p$  is the multiplication 
by a nonzero element of 
$\mathbb{Q}_p$.
For this reason, if $\alpha$ is a topological 
automorphism of   $\mathbb{Q}_p$, then
 either $\mathrm{Ker}(I+\alpha)=\{0\}$ or 
$\mathrm{Ker}(I+\alpha)=\mathbb{Q}_p$.
Applying Theorem \ref{th3} and Lemma \ref{lem3} to the group $\mathbb{Q}_p$, 
we get the following statement.

\begin{proposition}\label{pr1} Consider the group $\mathbb{Q}_p$  and  
let  $\alpha$ be a topological automorphism of $\mathbb{Q}_p$. 
Let $\xi_1$ and $\xi_2$ be
independent random variables with values in   $\mathbb{Q}_p$ and distributions
$\mu_1$ and $\mu_2$ with nonvanishing  characteristic functions. 
Assume that   the conditional  distribution of the linear form $L_2 = \xi_1 + 
\alpha\xi_2$ given $L_1 = \xi_1 + \xi_2$ is symmetric. Then the following 
statements are valid.
\renewcommand{\labelenumi}{\upshape\arabic{enumi}.}
\begin{enumerate}
\item	
If $\alpha\ne-I$, then
$\mu_1$ and $\mu_2$ are degenerate distributions. 
\item
If $\alpha=-I$, then $\mu_1=\mu_2=\omega$, where $\omega$ 
is an arbitrary distribution.
\end{enumerate}
\end{proposition}

\section{Heyde theorem  for the groups $\mathbb{R}^n\times G$, 
where $G$ is a totally \mbox{disconnected} locally  compact Abelian 
group  consisting of compact 
\mbox{elements}}

In this section we prove an analogue of Heyde's theorem for the group  
$X=\mathbb{R}^n\times G$, where $G$ is a totally disconnected locally  
compact Abelian group  consisting of compact elements and containing no 
elements of order  2. 
Denote by $x=(t, g)$, where $t\in \mathbb{R}^n$, $g\in G$, elements of the group 
$X$. Let $H$ be the character group of the group    $G$. Then the   group  $Y$ 
is topologically isomorphic to the group  $\mathbb{R}^n\times H$. Denote by 
$y=(s, h)$, where $s\in \mathbb{R}^n$, $h\in H$, elements of the group $Y$. 
Let  $\alpha$ be a topological automorphism of the group
$X$.   Since $\mathbb{R}^n$ is the connected component of the 
zero of the group  $X$, we have  $\alpha(\mathbb{R}^n)=\mathbb{R}^n$.  Since 
$G$ is the subgroup consisting of all compact elements of the group  $X$, we 
have $\alpha(G)=G$. Hence  $\alpha$ acts on elements of   $X$ as follows:  
$\alpha(t, g)=(\alpha_{{\mathbb{R}^n}} t, \alpha_{G} g)$, $t\in \mathbb{R}^n$, 
$g\in G$.  
In so doing, $\alpha_{\mathbb{R}^n}$ is an invertible linear operator in   
$\mathbb{R}^n$. Put $\mathcal{A}=\alpha_{\mathbb{R}^n}$.    Thus, 
$\alpha(t, g)=(\mathcal{A} t, \alpha_{G} g)$. We will write    $\alpha$   
in the form 
$\alpha=(\mathcal{A}, \alpha_{G})$  and will write  the adjoint automorphism
$\widetilde\alpha$ in the form $\widetilde\alpha=(\widetilde{\mathcal{A}}, 
\widetilde\alpha_{G})$.

Let us prove the following characterization theorem which can be considered 
as an analogue of Heyde's theorem for two independent random  variables for 
the group    $X=\mathbb{R}^n\times G$.
\begin{theorem}\label{th2} Let $X=\mathbb{R}^n\times G$, where $G$ is a totally 
disconnected  locally compact Abelian group  consisting of compact elements  
and containing no elements of order  $2$. Let  $\alpha$ be a topological 
automorphism 
of the group   $X$. Put $K=\mathrm{Ker}(I+\alpha)$. Let $\xi_1$ and $\xi_2$ be
independent random variables with values in   $X$ and distributions
$\mu_1$ and $\mu_2$ with nonvanishing  characteristic functions. If   
the conditional  distribution of the linear form $L_2 = \xi_1 + \alpha\xi_2$ given 
$L_1 = \xi_1 + \xi_2$ is symmetric, then there is an $\mathcal{A}$-invariant
 subspace
$F$ of  $\mathbb{R}^n$ such that $F\cap K=\{0\}$ and
$\mu_j=\gamma_j*\omega*E_{x_j}$, where $\gamma_j$ is a symmetric Gaussian 
distribution in  $F$, $\omega$ is a distribution supported in    
$K$, $x_j\in X$, $j=1, 2$.
\end{theorem}
To prove Theorem \ref{th2} we need some lemmas.
\begin{lemma}[{\protect\!\!\cite[Corollary 2.6]{H2}}]
\label{le25.1} 
Let    $\mathcal{A}$ be an invertible linear operator in the space   
$\mathbb{R}^n$.  
Put $K=\mathrm{Ker}(I+\mathcal{A})$ and suppose $K\ne\{0\}$.  Let 
$\xi_1$ and $\xi_2$ be
independent random vectors with values in  $\mathbb{R}^n$ and distributions 
$\mu_1$ and $\mu_2$. If   the conditional  distribution of 
the linear form $L_2 = \xi_1 + \mathcal{A}\xi_2$ given $L_1 = \xi_1 + \xi_2$ 
is symmetric, then there exists an  $\mathcal{A}$-invariant subspace $F$ of 
$\mathbb{R}^n$ 
satisfying the condition    
$F\cap K=\{0\}$ and such that some shifts $\lambda_j$  of the distributions 
$\mu_j$   are supported in the subspace $F\times K$. 
Furthermore, if $\eta_j$ are independent   random vectors with values in 
the space $\mathbb{R}^n$ 
and distributions  $\lambda_j$, then the conditional  distribution 
of the linear form
$N_2=\eta_1 + \mathcal{A}\eta_2$ given $N_1=\eta_1 + \eta_2$  is symmetric.
\end{lemma}

The following statement is an analogue of Heyde's theorem for  
independent random vectors  with values in the space  $\mathbb{R}^n$.

\begin{lemma}\label{A}  Let    
$\mathcal{A}$ be an invertible linear operator in the space   $\mathbb{R}^n$  
such that the number $-1$ is not its   eigenvalue.  Let $\xi_1$ and $\xi_2$ be
independent random vectors   with values in   $\mathbb{R}^n$ 
and distributions $\mu_1$ 
and $\mu_2$. If   the conditional  distribution of the linear form 
$L_2 = \xi_1 + \mathcal{A}\xi_2$ given $L_1 = \xi_1 + \xi_2$ is symmetric, 
then  $\mu_j$ are Gaussian distributions.
\end{lemma}
\begin{proof}
By Lemma \ref{lem8}, the linear forms $M_1=(I+\mathcal{A})\xi_1+2\mathcal{A}\xi_2$ 
and
$M_2=2\xi_1+(I+\mathcal{A})\xi_2$ are independent. Since the number $-1$ is not an
   eigenvalue of the operator $\mathcal{A}$, the operator $I+\mathcal{A}$ is invertible.
   Hence the coefficients of the linear forms $M_1$ and $M_2$ are
   invertible operators. By the well-known Ghurye--Olkin 
   theorem (\!\!\cite{GhurO}, see also \cite[Theorem 3.2.1]{KaLiRa}), $\mu_1$ 
and $\mu_2$ are Gaussian distributions\footnote{The original proof of 
the Ghurye--Olkin 
   theorem in the case of $n$ independent random variables is based on 
   solving the Skitovich--Darmois equation in the space $\mathbb{R}^m$ 
   (see \cite[\S 3.2]{KaLiRa}). It should be 
   noted that based on characterization of polynomials as the solutions 
   set of some functional equations,
J. M. Almira in \cite{A1} proposed a new approach to solving the Skitovich--Darmois 
functional equation in the
space $\mathbb{R}^m$ for arbitrary $n$ functions.}. 
\end{proof}

Let $K$ be a locally compact Abelian group with character group  $L$.
Consider the group $X=\mathbb{R}^m\times K$. 
The   
group  $Y$ is topologically isomorphic to the group  $\mathbb{R}^m\times L$. 
Denote by $(t, k)$,
$t\in \mathbb{R}^m$, $k\in K$, elements  of the group $X$ and by
 $(s, l)$,
$s\in \mathbb{R}^m$, $l\in L$, elements  of the group $Y$.
Put 
$$
B_r=\{s=(s_{1}, s_{2},\dots, s_{m})\in \mathbb{C}^m: |s_{j}|\le r, 
\ j=1, 2, \dots, m\}.
$$

\begin{lemma}[{\protect\!\!\cite[Proposition 2.13]{Febooknew}}]\label{lem5}
 Let
$\mu$ be a distribution on
the group $X=\mathbb{R}^m\times K$  with the characteristic function    
$\hat\mu(s, l)$, 
$s\in \mathbb{R}^m$, $l\in L$. Assume that the function $\hat\mu(s, 0)$, 
$s\in \mathbb{R}^m$, is extended to   $\mathbb{C}^m$  as an entire 
function in $s$. Then for any fixed   $l\in L$  the function $\hat\mu(s, l)$  
is also extended to  $\mathbb{C}^m$ as an entire function in $s$, the 
representation
$$
	\hat\mu(s, l)=\int\limits_{X} e^{i\langle t, s\rangle}(k,	l)d\mu(t, k)
	$$
	is valid for all $s\in \mathbb{C}^m$, $l\in L$,
and  for each $l\in L$ the inequality
\begin{equation}\label{16.05.6}
\max_{ s\in B_r}|\hat\mu(s, l)|\le \max_{s\in B_r}|\hat\mu(s, 0)| 
\end{equation}
 holds.
\end{lemma}
The following lemma is a generalization of Lemma 3.3 in \cite{H1}. In the proof    
we  follow 
the scheme of the proof of Theorem 2.1 in
\cite{H2}.
\begin{lemma}\label{lem10}  Let $X=\mathbb{R}^m\times K$, where $K$ is a 
locally compact Abelian group  containing no elements of order $2$.
Let  $\alpha$ be a topological automorphism of the group    $X$ of the 
form $\alpha(t, k)=(\mathcal{A}t, -k)$, where $\mathcal{A}$ is an invertible 
linear operator 
in the space   $\mathbb{R}^m$  such that the number $-1$ is not its   eigenvalue.  
Let $\xi_1$ and $\xi_2$ be
independent random variables with values in   $X$ and distributions 
$\mu_1$ and $\mu_2$ with nonvanishing  characteristic functions.
 If   the conditional  distribution of the linear form  $L_2 = \xi_1 + \alpha\xi_2$ 
 given $L_1 = \xi_1 + \xi_2$ is symmetric, then $\mu_j=\gamma_j*\omega$, 
 where $\gamma_j$ is a Gaussian distribution in  $\mathbb{R}^m$ and 
 $\omega$ is a distribution supported in $K$,   $j=1, 2$.
\end{lemma}
\begin{proof}  Denote by $L$ the character group of the group  $K$ and by 
 $y=(s, l)$, where $s\in \mathbb{R}^m$, $l\in L$, elements of 
the group $Y$.   By Lemma \ref{lem1}, the characteristic functions   
$\hat\mu_j(s, l)$   satisfy equation (\ref{11.04.1}) which takes the form
\begin{multline}\label{14.04.2}
\hat\mu_1(s_1+s_2, l_1+l_2)\hat\mu_2(s_1+\widetilde{\mathcal{A}} s_2, l_1- l_2)\\=
\hat\mu_1(s_1-s_2, l_1-l_2)\hat\mu_2(s_1-\widetilde{\mathcal{A}} s_2, l_1+l_2), 
\quad  s_j\in \mathbb{R}^m, \ l_j\in L.
\end{multline}
Substituting  $l_1=l_2=0$ in equation (\ref{14.04.2})  we obtain
\begin{equation}\label{25.07.1}
\hat\mu_1(s_1+s_2, 0)\hat\mu_2(s_1+\widetilde{\mathcal{A}} s_2, 0)=
\hat\mu_1(s_1-s_2, 0)\hat\mu_2(s_1-\widetilde{\mathcal{A}} s_2, 0), \quad  
s_j\in \mathbb{R}^m.
\end{equation}
Taking into account Lemmas \ref{lem1} and \ref{A}, equation 
(\ref{25.07.1}) implies that
$\hat\mu_j(s, 0)$, $j=1, 2$, are the characteristic functions of Gaussian 
distributions in the space $\mathbb{R}^m$, 
i.e.,
\begin{equation}\label{16.05.4}
\hat\mu_j(s, 0)=\exp\{-\langle A_j s, s\rangle+i\langle b_j, s\rangle\}, 
\quad s\in \mathbb{R}^m,
\end{equation}
where $A_j$ are   symmetric positive semidefinite $m\times m$ matrices, 
$b_j\in \mathbb{R}^m$, $j=1, 2$. Substituting (\ref{16.05.4}) into  
(\ref{25.07.1}) we obtain that $b_1+\mathcal{A}b_2=0$. This implies that
\begin{equation}\label{116.05.4}
\exp\{i\langle b_1, s_1+s_2\rangle\}
\exp\{i\langle b_2, s_1+\widetilde{\mathcal{A}} s_2\rangle\}
=
\exp\{i\langle b_1, s_1-s_2\rangle\}
\exp\{i\langle b_2, s_1-\widetilde{\mathcal{A}} s_2\rangle\}, \quad  
s_j\in \mathbb{R}^m.
\end{equation}
It follows from Lemma  \ref{lem1} and (\ref{116.05.4}) that we can
 replace the distributions $\mu_j$ with their shifts $\mu_j*E_{-b_j}$ and 
 suppose that $b_1=b_2=0$ in (\ref{16.05.4}), 
i.e.,
\begin{equation}\label{16.05.5}
\hat\mu_j(s, 0)=\exp\{-\langle A_j s, s\rangle\}, \quad s\in \mathbb{R}^m, 
\quad j=1, 2.
\end{equation}
By Lemma \ref{lem5},  for each fixed $l\in L$ the functions $\hat\mu_j(s, l)$ 
are extended to    $\mathbb{C}^m$  as  entire functions in $s$. It is easy to 
see that equation   (\ref{14.04.2}) is valid for all $s_j\in \mathbb{C}^m$, $l_j\in L$.

Let us check that the functions $\hat\mu_j(s, l)$ do not vanish 
in  $\mathbb{C}^m\times L$. Consider equation (\ref{14.04.2}) for 
$s_j\in \mathbb{C}^m$, $l_j\in L$, and put   $l_1=l_2=l$ 
in it.   Taking into account (\ref{16.05.5}), we rewrite the resulting equation in 
the form
\begin{multline} \label{114.04.4}
\hat\mu_1(s_1+s_2, 2l)\exp\{-\langle A_2 (s_1+\widetilde{\mathcal{A}} s_2),  
s_1+\widetilde{\mathcal{A}} s_2 \rangle\}\\=
\exp\{-\langle A_1 (s_1- s_2),  s_1-  s_2  \rangle\}
\hat\mu_2(s_1-\widetilde{\mathcal{A}} s_2, 2l), \quad  s_j\in \mathbb{C}^m, \ l\in L.
\end{multline}
It follows from Lemma \ref{lem5} that the functions $\hat\mu_j(s, l)$, $j=1, 2$,
are continuous on  $\mathbb{C}^m\times L$. Since the group $K$ contains no 
elements of order 2, Lemma  \ref{lem9} implies that the subgroup $L^{(2)}$ is dense in $L$.
Hence from equation (\ref{114.04.4}) we find
\begin{multline}\label{14.04.4}
\hat\mu_1(s_1+s_2, l)\exp\{-\langle A_2 (s_1+\widetilde{\mathcal{A}} s_2),  s_1+\widetilde{\mathcal{A}} s_2 \rangle\}\\=
\exp\{-\langle A_1 (s_1- s_2),  s_1-s_2 \rangle\}
\hat\mu_2(s_1-\widetilde{\mathcal{A}} s_2, l), \quad  s_j\in \mathbb{C}^m, \ l\in L.
\end{multline}

Suppose that  $\hat\mu_1(s_0, l_0)=0$ for some $s_0\in \mathbb{C}^m$, $l_0\in L$. 
Since 
the number $-1$ is not an eigenvalue of the operator    $\mathcal{A}$, the operator 
$\mathcal{A}+I$ is invertible. Hence  the operator $\widetilde{\mathcal{A}}+I$ 
is also invertible. Substitute   
$s_1=\widetilde{\mathcal{A}}(\widetilde{\mathcal{A}}+I)^{-1}s_0$, 
$s_2=(\widetilde{\mathcal{A}}+I)^{-1}s_0$, and 
$l=l_0$ in equation (\ref{14.04.4}). We have $s_1+s_2=s_0$ and  
$s_1-\widetilde{\mathcal{A}} s_2=0$.
Since $\hat\mu_2(0, l_0)\ne 0$,  the right-hand 
side in the resulting expression is nonzero, whereas the left-hand side 
is equal to the zero. This contradiction implies that the function 
$\hat\mu_1(s, l)$ does not vanish in $\mathbb{C}^m\times L$.   Similarly, if 
$\hat\mu_2(s_0, l_0)=0$ for some $s_0\in \mathbb{C}^m$, $l_0\in L$, then substituting  $s_1=(\widetilde{\mathcal{A}}+I)^{-1}s_0$, 
$s_2=-(\widetilde{\mathcal{A}}+I)^{-1}s_0$, and $l_0\in L$ in 
equation (\ref{14.04.4})  
and taking into account that $s_1+s_2=0$,  $s_1-\widetilde{\mathcal{A}} s_2=s_0$,  
and $\hat\mu_1(0, l_0)\ne 0$, we get a contradiction. Hence the 
function $\hat\mu_2(s, l)$ also does not 
vanish in $\mathbb{C}^m\times L$. Thus, we have the representation
$$
\hat\mu_j(s, l)=\exp\{Q_j(s, l)\}, \quad  s  \in \mathbb{C}^m, \ l\in L,
$$
where for a fixed   $l\in L$ each of the functions
 $Q_j(s, l)$, $j=1, 2$, is entire in   $s$ in $\mathbb{C}^m$.
 
Fix $l\in L$ and consider the restriction of the function  
$Q_j(s, l)$, $j=1, 2$, to any complex plane in $\mathbb{C}^m$ passing through 
the zero. 
 Taking into account Hadamard's theorem on the representation of entire 
 functions of finite order and  Lemma \ref{lem5},  it follows from
  (\ref{16.05.6})
and (\ref{16.05.5}) that this restriction is a polynomial of degree at most 2.
 Hence  the function  
$Q_j(s, l)$ is a polynomial  in coordinate of 
$s$ of degree at most 2. Thus, we have the representation
\begin{equation}\label{16.05.7}
\hat\mu_j(s,l)=\exp\{\langle A_j(l) s, s\rangle+\langle b_j(l), s\rangle+c_j(l)\}, 
\quad  s\in \mathbb{C}^m, \ l\in L,
\end{equation}
where $A_j(l)$ is a    symmetric complex $m\times m$ matrix, 
$b_j(l)\in \mathbb{C}^m$, $c_j(l)\in \mathbb{C}$,  $j=1, 2$. 
We will prove that $A_j(l)=A_j$ and $b_j(l)=0$ for all $l\in L$, $j=1, 2$. 

Substitute  
$l_1=l_2=l$ in equation (\ref{14.04.2}). 
Since the subgroup $L^{(2)}$ is dense in $L$, from the resulting equation we get
\begin{equation}\label{n14.04.2}
\hat\mu_1(s_1+s_2, l)\hat\mu_2(s_1+\widetilde{\mathcal{A}} s_2, 0)\\=
\hat\mu_1(s_1-s_2, 0)\hat\mu_2(s_1-\widetilde{\mathcal{A}} s_2, l), 
\quad  s_j\in \mathbb{R}^m, \ l\in L.
\end{equation}
Substituting (\ref{16.05.7}) into equation (\ref{n14.04.2}), 
we find
\begin{multline}\label{16.05.8}
\langle A_1(l)(s_1+s_2), s_1+s_2\rangle+\langle A_2(0)(s_1+
\widetilde{\mathcal{A}} s_2), s_1+\widetilde{\mathcal{A}} s_2\rangle\\
=\langle A_1(0)(s_1-s_2), s_1-s_2\rangle+
\langle A_2(l)(s_1-\widetilde{\mathcal{A}} s_2), 
s_1-\widetilde{\mathcal{A}} s_2\rangle,\quad    
s_1, s_2\in \mathbb{R}^m, \ l\in L,
\end{multline}
and
\begin{equation}\label{17.05.3}
\langle b_1(l), s_1+s_2\rangle+\langle b_2(0), 
s_1+\widetilde{\mathcal{A}}s_2\rangle=
\langle b_1(0), s_1-s_2\rangle+\langle b_2(l), 
s_1-\widetilde{\mathcal{A}}s_2\rangle,\quad s_1, s_2\in \mathbb{R}^m, \ l\in L.
\end{equation}

Equation  (\ref{16.05.8}) implies the equalities
\begin{equation}\label{17.05.2}
A_1(l)+A_2(0)=A_1(0)+A_2(l),   \quad l\in L,
\end{equation}
 and
\begin{equation}\label{16.05.9}
A_1(l)+\mathcal{A}A_2(0)\widetilde{\mathcal{A}}=
A_1(0)+\mathcal{A}A_2(l)\widetilde{\mathcal{A}},   \quad l\in L.
\end{equation}
Substituting  $l_1=l$,  $l_2=0$ in equation (\ref{14.04.2})  and taking 
into account (\ref{16.05.7}), from the obtained equation we find
\begin{multline}\label{16.05.10}
\langle A_1(l)(s_1+s_2), s_1+s_2\rangle+\langle A_2(l)
(s_1+\widetilde{\mathcal{A}}s_2), s_1+\widetilde{\mathcal{A}}s_2\rangle\\
=\langle A_1(l)(s_1-s_2), s_1-s_2\rangle+\langle A_2(l)
(s_1-\widetilde{\mathcal{A}}s_2), s_1-\widetilde{\mathcal{A}}s_2\rangle,\quad    
s_1, s_2\in \mathbb{R}^m,   
\ l\in L.
\end{multline}
From (\ref{16.05.10}) we get
\begin{equation}\label{16.05.11}
A_1(l)+\mathcal{A}A_2(l)=0,   \quad l\in L.
\end{equation}
We find from (\ref{16.05.5}) and (\ref{16.05.7}) that $A_j(0)=-A_j$, $j=1, 2$. 
Since the operator $\widetilde{\mathcal{A}}+I$ is invertible,    
 (\ref{16.05.9}) and (\ref{16.05.11}) imply that $A_1(l)=-A_1$,   $l\in L$, 
 and we obtain from  (\ref{17.05.2}) that $A_2(l)=-A_2$,   $l\in L$.
Thus, we proved that
\begin{equation}\label{17.05.5}
A_1(l)=-A_1, \quad A_2(l)=-A_2, \quad l\in L.
\end{equation}

It follows from (\ref{17.05.3}) that
\begin{equation}\label{17.05.4}
b_1(l)+b_2(0)=b_1(0)+b_2(l), \quad b_1(l)+ \mathcal{A} b_2(0)=-b_1(0)- 
\mathcal{A} b_2(l), \quad l\in L.
\end{equation}
 Observe that (\ref{16.05.5}) and (\ref{16.05.7}) imply $b_1(0)=b_2(0)=0$. 
 Taking this into account and taking into account that the operator    
 $\mathcal{A}+I$   is invertible, (\ref{17.05.4})  implies that
 \begin{equation}\label{10.05.1}
b_1(l)=b_2(l)=0, \quad l\in L.
\end{equation}

In view of (\ref{17.05.5}) and (\ref{10.05.1}), we get from
(\ref{16.05.7}) that the characteristic functions $\hat\mu_j(s, l)$ can be 
represented in the form
\begin{equation}\label{20.05.3}
\hat\mu_j(s, l)=\exp\{-\langle A_j s, s\rangle+c_j(l)\}, \quad  
s\in \mathbb{R}^m, \ l\in L, \   j=1, 2.
\end{equation}

Substituting      $s_1=s_2=0$ in equation  (\ref{n14.04.2}),    we get   
\begin{equation}\label{nn17.05.7}
\hat\mu_1(0,l)=\hat\mu_2(0,l), \quad l\in L.
\end{equation}
Substituting $s=0$ into (\ref{20.05.3}) and taking into account 
(\ref{nn17.05.7}), we obtain
\begin{equation}\label{17.05.7}
\hat\mu_1(0,l)=\hat\mu_2(0,l)=\exp\{c_1(l)\}=\exp\{c_2(l)\},\quad  l\in L.
\end{equation}
Denote by $\gamma_j$ a symmetric Gaussian distribution in the space 
$\mathbb{R}^m$ 
with the characteristic function
\begin{equation}\label{19.05.2}
\hat\gamma_j(s)=\exp\{-\langle A_js, s\rangle\}, \quad s\in \mathbb{R}^m, 
\ j=1, 2.
\end{equation}
In view of  (\ref{nn17.05.7}), denote by $\omega$  
a distribution on the group    $K$ with the characteristic function
\begin{equation}\label{20.05.5}
\hat\omega(l)=\hat\mu_1(0,l)=\hat\mu_2(0,l), \quad l\in L.
\end{equation}
The representation
$$
\hat\mu_j(s, l)=\hat\gamma_j(s)\hat\omega(l), \quad s\in 
\mathbb{R}^m, \ l\in L, \ j=1, 2,
$$
follows from   (\ref{20.05.3}) and (\ref{17.05.7})--(\ref{20.05.5}). 
This implies that    
$\mu_j=\gamma_j*\omega$,  $j=1, 2$. The lemma is proved.  \end{proof}

\textit{Proof of Theorem $\ref{th2}$}  Let  $\alpha=(\mathcal{A}, \alpha_{G})$.   
Put $$K_1=K\cap \mathbb{R}^n=\mathrm{Ker}(I+\mathcal{A}), \quad 
K_2=K\cap G=\mathrm{Ker}(I+\alpha_{G}).$$  It is obvious that $K=K_1\times K_2$.  
By Lemma \ref{lem1}, the characteristic functions $\hat\mu_j(s, h)$   
satisfy equation (\ref{11.04.1})  which takes the form
\begin{multline}\label{13.04.2}
\hat\mu_1(s_1+s_2, h_1+h_2)\hat\mu_2(s_1+\widetilde{\mathcal{A}} s_2, 
h_1+\widetilde\alpha_{G} h_2)\\=
\hat\mu_1(s_1-s_2, h_1-h_2)\hat\mu_2(s_1-\widetilde{\mathcal{A}} s_2, 
h_1-\widetilde\alpha_{G} h_2), \quad  s_j\in \mathbb{R}^n, \ h_j\in H.
\end{multline}
Denote by $\pi_j$ a distribution in the space   $\mathbb{R}^n$ with 
the characteristic function  $\hat\pi_j(s)=\hat\mu_j(s, 0)$, 
$s\in \mathbb{R}^n$, $j=1, 2$.
Setting $h_1=h_2=0$ in equation (\ref{13.04.2}), we get that the characteristic 
functions  $\hat\pi_j(s)$ satisfy the equation
\begin{equation}\label{26.07.1}
\hat\pi_1(s_1+s_2)\hat\pi_2(s_1+\widetilde{\mathcal{A}} s_2)=
\hat\pi_1(s_1-s_2)\hat\pi_2(s_1-\widetilde{\mathcal{A}} s_2), \quad  
s_j\in \mathbb{R}^n.
\end{equation}
In view of Lemma   \ref{lem1} and (\ref{26.07.1}), it follows from 
Lemma \ref{le25.1} that there exist  an  $\mathcal{A}$-invariant subspace   
$F$ 
of $\mathbb{R}^n$ satisfying the condition  $F\cap K_1=\{0\}$   and vectors
$t_j\in \mathbb{R}^n$ such that the distributions  $\rho_j=\pi_j*E_{-t_j}$ 
are supported in the subspace  $F\times K_1$. Moreover, if $\zeta_j$ are 
independent random vectors with values in the space $\mathbb{R}^n$ and
 distributions  $\rho_j$, then the conditional distribution of the linear form
$P_2 = \zeta_1 + \mathcal{A} \zeta_2$ given
  $P_1 = \zeta_1 + \zeta_2$  is symmetric.

Denote by $\omega_j$ a distribution on the group $G$ with the characteristic 
function $\hat\omega_j(h)=\hat\mu_j(0, h)$,  $h\in H$, $j=1, 2$. Substituting 
$s_1=s_2=0$ in (\ref{13.04.2}), we obtain that  the characteristic functions 
$\hat\omega_j(h)$ satisfy the equation
\begin{equation}\label{14.04.1}
\hat\omega_1(h_1+h_2)\hat\omega_2(h_1+\widetilde\alpha_{G} h_2)=
\hat\omega_1(h_1-h_2)\hat\omega_2(h_1-\widetilde\alpha_{G} h_2), \quad    
h_j\in H.
\end{equation}
Taking into account Lemma   \ref{lem1} and (\ref{14.04.1}), it follows from 
Theorem \ref{th3}  applying to the group  $G$, that there exist  elements
$g_j\in G$ such that the distribution   $\omega=\omega_1*E_{-g_1}=\omega_2*E_{-g_2}$ 
is supported in the subgroup $K_2$. Moreover,  if $\tau_j$  are independent 
identically distributed random variables with values in   $G$  and distribution  
$\omega$, then    the conditional  distribution of the linear form 
$Q_2 = \tau_1 +\alpha_{G}\tau_2$ given $Q_1 = \tau_1 + \tau_2$ is symmetric.

Consider the distributions $\lambda_j=\mu_j*E_{-t_j-g_j}$, $j=1, 2$. Let $\eta_j$ 
be
independent random variables with values in   $X$ and distributions   $\lambda_j$. 
From the above it follows that   the conditional  distribution of the linear form
$N_2=\eta_1 + \alpha\eta_2$ given $N_1=\eta_1 + \eta_2$    is symmetric  and   
$\sigma(\lambda_j)\subset F\times K_1\times K_2=F\times K$, $j=1, 2$. Therefore 
we can consider $\eta_j$ as
 independent random variables with values in the subgroup $F\times K$. In view 
 of  $K=\mathrm{Ker}(I+\alpha)$, the restriction of the topological automorphism   
 $\alpha$  
 to the subgroup $K$ coincides with $-I$. Hence  the restriction of the topological 
 automorphism  $\alpha$ to the subgroup $F\times K$ is of the form 
 $(\mathcal{A}_{F}, -I)$. Obviously, the   group     $K$ contains no elements 
 of order 2. 
It follows from $F\cap K_1=\{0\}$ that the number $-1$ is not an 
eigenvalue of the operator $\mathcal{A}_{F}$.  
 The statement of the theorem  follows from Lemma \ref{lem10}  
 applying to the group $F\times K$. 
 \hfill$\Box$

\section{Heyde theorem  for some locally compact Abelian groups containing 
an element of order 2}

Denote by $\mathbb{Z}(2)=\{0, 1\}$ the   group  of the integers 
modulo $2$. Consider  the group   
$\mathbb{R}\times \mathbb{Z}(2)$ and denote  by 
$(t,   m)$, where $t\in \mathbb{R}$,   $m\in \mathbb{Z}(2)$, its elements.
Elements of the character group   of the group $\mathbb{R}\times \mathbb{Z}(2)$,
which is topologically isomorphic 
 to $\mathbb{R}\times \mathbb{Z}(2)$,  
 we denote  by $(s, n)$, where $s\in \mathbb{R}$,   $n\in \mathbb{Z}(2)$. 
Let $\alpha\in \mathrm{Aut}(\mathbb{R}\times \mathbb{Z}(2))$. 
It is obvious that
 $\alpha$ is of the form
 $\alpha(t, m)=(a t, m)$, $t\in \mathbb{R}$,   $m\in \mathbb{Z}(2)$,  
 where $a\in \mathbb{R}$, $a\ne 0$.

Let $\sigma\ge 0$,  $\sigma'\ge 0$ and  
$\beta, \beta', \varkappa\in \mathbb{R}$. 
According to \cite{JMAA2024}, see also \cite[Definition 11.2]{Febooknew}, we say that  a probability distribution $\mu$ on the group $\mathbb{R}\times \mathbb{Z}(2)$  
belongs to the class $\Theta$  if the characteristic function $\hat\mu(s, n)$, $s\in \mathbb{R}$,   
$n\in \mathbb{Z}(2)$,
is represented in the form 
\begin{equation*}\label{21.01.1}  
\hat\mu(s, n) = \begin{cases}\exp\{-\sigma s^2+i\beta s\}  &\text{\ if\ }\ 
\ s\in \mathbb{R}, 
\ n=0,\\
\varkappa\exp\{-\sigma' s^2+i\beta's\}  &\text{\ if\ }\ \ s\in \mathbb{R}, 
\   n=1,
\\
\end{cases}
\end{equation*}
  where either \begin{equation*}\label{08.01.8}
0<\sigma'<\sigma,  \quad
0<|\varkappa|\le\sqrt\frac{\sigma'}{\sigma}\exp\left\{-\frac{(\beta-\beta')^2}
{4(\sigma-\sigma')}\right\} 
\end{equation*}
or
\begin{equation*}\label{08.01.9}
\sigma=\sigma', \quad\beta=\beta', \quad  |\varkappa|\le 1.
\end{equation*} are satisfied. 

We note that the class $\Theta$ is a semigroup with respect to convolution.
Arithmetic of the semigroup $\Theta$, namely, 
the class of infinitely divisible distributions, 
the class  of indecomposable distributions, and the class  of 
distributions which  have no indecomposable factors, was studied in \cite{JMPAG2024}.

Recall that for a torsion Abelian group $F$ and a prime $p$ the subgroup consisting 
of all elements of $F$ whose order is a power of $p$ is called the $p$-component of $F$.
Consider a group $X$ of the form $X=\mathbb{R}\times F$,
where $F$ is a discrete torsion Abelian group such that its 2-component is isomorphic to 
$\mathbb{Z}(2)$.
Since a torsion Abelian group is isomorphic to a weak direct product of
its $p$-components (\!\!\cite[(A.3)]{Hewitt-Ross}), the group $F$ is topologically 
isomorphic to a group of the form $\mathbb{Z}(2)\times G$, 
where $G$ is a discrete torsion Abelian group containing no elements of order 2. 
So, we may  assume, without loss of generality, that $X=\mathbb{R}\times \mathbb{Z}(2)\times G$.
Obviously, the subgroups $\mathbb{R}$, $\mathbb{Z}(2)$, and $G$ are characteristic.
Let $\alpha$ be a topological automorphism of 
the group $X$. 
Then $\alpha$ acts on the elements of the group $X$ as follows: 
$\alpha(t, m, g)=(at, m, \alpha_Gg)$, $t\in \mathbb{R}$, $m\in\mathbb{Z}(2)$, 
$g\in G$, where $a\in \mathbb{R}$, $a\ne 0$, and we will write $\alpha=(a, I, \alpha_G)$.

The following theorem was proved in \cite{JMAA2024}.

\begin{heyde2*}Assume that $X=\mathbb{R}\times \mathbb{Z}(2)\times G$, 
where $G$ is a finite Abelian group containing no elements of order  $2$.
Let  $\alpha$ be a topological automorphism of 
the group $X$. Then $\alpha$ can be written in the form
$\alpha=(a, I, \alpha_G)$. Set $K=\mathrm{Ker}(I+\alpha_{G})$.
Let $\xi_j$, $j=1, 2$, be
independent random variables with values in   $X$ and distributions
$\mu_j$ with nonvanishing characteristic functions. Assume that the conditional  distribution of the linear form $L_2 = \xi_1 + \alpha\xi_2$ given $L_1 = \xi_1 + \xi_2$ is symmetric. Then the following statements are true.
\renewcommand{\labelenumi}{\rm{\Roman{enumi}.}}
\begin{enumerate}
  
\item	

If  $a \ne-1$, then $\mu_j=\gamma_j*\omega_j*E_{g_j}$, where $\gamma_j$ are distributions of the class $\Theta$ on the subgroup   
$\mathbb{R}\times \mathbb{Z}(2)$, $\omega_j$ are distributions on the subgroup    
$\mathbb{Z}(2)\times K$ and either 
$\omega_1=\omega_2*\vartheta_2$ or   $\omega_2=\omega_1*\vartheta_1$, 
 where  $\vartheta_j$ are distributions on $\mathbb{Z}(2)$ and $g_j\in G$. 

\item

If  $a=-1$, then 
$\mu_j=\omega_j*E_{x_j}$, where  $\omega_j$ are distributions on  
the subgroup $\mathbb{R}\times\mathbb{Z}(2)\times K$  and either 
$\omega_1=\omega_2*\vartheta_2$ or   $\omega_2=\omega_1*\vartheta_1$, 
 where  $\vartheta_j$ are distributions on $\mathbb{Z}(2)$ and $g_j\in G$. 
\end{enumerate}
\end{heyde2*}

The proof of Theorem B is based on Theorem A
for a finite Abelian group $G$ containing no elements of order  $2$. 
Moreover, an analysis of 
the proof of Theorem B shows that it is true 
if $X=\mathbb{R}\times \mathbb{Z}(2)\times G$,
where $G$ is a 
discrete Abelian group containing no elements of order 2, 
provided that Theorem A holds for $G$ and $G$ is a characteristic 
subgroup of $X$.

Let $G$ be a discrete torsion Abelian group 
containing no elements of order $2$. By Theorem \ref{th3}, 
Theorem A holds for $G$. It is obvious that $G$ is a characteristic 
subgroup of $X$. This implies
the following statement.

\begin{theorem}
Theorem $\mathrm{B}$ remains valid if $G$ is a discrete torsion Abelian group 
containing no elements of order $2$. 
\end{theorem}

\section{Some problems}

Recall that a distribution  $\gamma$  on a locally compact Abelian 
group $X$ is called Gaussian
(see \cite[Chapter IV, \S 6]{Pa})
if its characteristic function has the form
$$
\hat\gamma(y)= (x,y)\exp\{-\varphi(y)\}, \quad  y\in Y,
$$
where $x \in X$, and $\varphi(y)$ is a continuous nonnegative function
on the group $Y$
 satisfying the equation
 $$
\varphi(u + v) + \varphi(u
- v) = 2[\varphi(u) + \varphi(v)], \quad u,  v \in
Y.
$$
In particular, all degenerate distributions are Gaussian. Note that 
the support of  any Gaussian distribution is a coset of a connected 
subgroup of the group $X$. This implies that on a totally disconnected 
locally compact Abelian group Gaussian distributions are degenerated.

Let  $X$ be either a totally disconnected locally compact Abelian group 
containing no elements of order  $2$   or a group of the form 
$X=\mathbb{R}^n\times G$, 
where $G$ is a totally disconnected locally  compact Abelian group  
consisting of compact elements and containing no elements of order  
$2$. Let
  $\alpha$ be a topological automorphism of the group $X$. 
  Put $K=\mathrm{Ker}(I+\alpha)$. Let $\xi_1$ and $\xi_2$ be
independent random variables with values in   $X$ and distributions
$\mu_1$ and $\mu_2$ with nonvanishing  characteristic functions. 
Assume that the conditional  distribution of the 
linear form $L_2 = \xi_1 + \alpha\xi_2$ given $L_1 = \xi_1 + \xi_2$ 
is symmetric. Taking into account that the support of  
any Gaussian distribution is a coset of a connected 
subgroup of the group $X$, it follows  from  Theorems \ref{th3} and  \ref{th2} 
that then 
$\mu_j$  are   convolutions  of Gaussian 
distributions on  $X$ and a distribution supported in    $K$.
A similar result was proved in \cite{F_solenoid} in the case when 
$X$ is an arbitrary ${\bm a}$-adic   solenoid\footnote{The definition of an 
${\bm a}$-adic   solenoid and its properties see, e.g., 
\cite[(10.12), (10.13)]{Hewitt-Ross}.}
containing no elements of order  $2$.
 In view of the above, we formulate the following general problem.

\begin{problem}\label{p1} To describe locally compact Abelian groups $X$ 
containing no elements of order  $2$ for which the following statement is valid:

\textit{Let   $\alpha$ be a topological automorphism of the group   $X$. 
Put $K=\mathrm{Ker}(I+\alpha)$. Let $\xi_1$ and $\xi_2$ be
independent random variables with values in   $X$ and distributions
$\mu_1$ and $\mu_2$ with nonvanishing  characteristic functions. 
If   the conditional  distribution of the linear form $L_2 = \xi_1 + 
\alpha\xi_2$ given $L_1 = \xi_1 + \xi_2$ is symmetric, then
 $\mu_j$  are   convolutions  of Gaussian distributions on  $X$ and a 
 distribution supported in    $K$}.
 \end{problem}
 
Let us formulate some particular
  cases of Problem \ref{p1}.
 
\begin{problem}\label{p2} Does Theorem \ref{th2} remain true if 
we omit the condition: the group $G$ consists of compact elements.
 \end{problem}
 
 \begin{problem}\label{p3} Give an example of a locally compact Abelian 
 group $X$ containing no elements of order  $2$ for which the 
 following statement is true: 
 
 \textit{There exist independent random variables 
 $\xi_1$ and $\xi_2$ with values in $X$ and distributions
$\mu_1$ and $\mu_2$ with nonvanishing  characteristic functions and a 
 topological automorphism $\alpha$ of the group   $X$ such that 
  the conditional  distribution of the linear form $L_2 = \xi_1 + 
\alpha\xi_2$ given $L_1 = \xi_1 + \xi_2$ is symmetric while
 $\mu_j$  cannot be represented as
    convolutions  of Gaussian distributions on  $X$ and a 
 distribution supported in    $K$, where $K=\mathrm{Ker}(I+\alpha)$}.
 \end{problem}

\noindent\textbf{Funding} The author  declares that no funds or grants  
 were
received during the preparation of this manuscript.

\medskip

\noindent\textbf{Data Availability} No datasets were generated or analysed 
during the current study.

\bigskip

\noindent\textbf{\Large Declarations}

\bigskip

\noindent\textbf{Ethical Approval} Not applicable.

\medskip

\noindent\textbf{Competing interests} The author  declares no competing interests.

\medskip

\noindent B. Verkin Institute for Low Temperature Physics and Engineering\\
of the National Academy of Sciences of Ukraine\\
47, Nauky ave, Kharkiv, 61103, Ukraine

\medskip

\noindent e-mail:    gennadiy\_f@yahoo.co.uk

\end{document}